\documentclass[11pt]{amsart}
\usepackage{pxfonts}
\usepackage{enumitem}
\usepackage{epigraph}
\usepackage[T1]{fontenc}
\usepackage[pdftitle= {Thetaregularity},pdfauthor={Sofia Tirabassi}]{hyperref}
\usepackage{amscd, amssymb, amsmath, wasysym, yfonts}
\usepackage{graphicx}
\usepackage[all]{xy}
\usepackage{url}
\usepackage[usenames, dvipsnames]{color}
\usepackage{hyperref}
 \hypersetup{
  colorlinks=true,
   citecolor=Plum,
   linkcolor=OliveGreen,
   urlcolor=blue}

   \usepackage{enumitem}
   \usepackage{chngcntr} 
\counterwithout{subsection}{section}

\usepackage{amsrefs}
\usepackage[margin=1.00in]{geometry}
\theoremstyle{plain}
\newtheorem{thm}{Theorem}

\newtheorem{theorem}{Theorem}[subsection]
\newtheorem{prop}[theorem]{Proposition}

\newtheorem{Question}{Question}
\theoremstyle{definition}

\newtheorem{rmk}[theorem]{Remark}

\newtheorem{ex}[theorem]{Example}
\newtheorem*{thm*}{Theorem}
\newtheorem*{ex*}{Example}

\newcommand\sO{{\mathcal O}}

\newcommand\J{{\mathcal J}}
\newcommand\I{{\mathcal I}}

\usepackage[colorinlistoftodos]{todonotes}

\title
{Theta-regularity and log--canonical threshold}

\author{Morten \O ygarden}
\address{Department of Mathematics\\ University of Bergen, All\'egaten 41, Bergen, Norway}
\email{\url{morten.oygarden@student.uib.no}}

\author{Sofia Tirabassi}

\address{Department of Mathematics\\ University of Bergen, All\'egaten 41, Bergen, Norway}
\email{\url{sofia.tirabassi@uib.no}}

\begin{document} \maketitle
\begin{abstract}
 We show that an inequality, proven by K\"uronya--Pintye, which governs the behavior of the log--canonical threshold of an ideal over $\mathbb{P}^n$ and that of its Castelnuovo--Mumford regularity, can be applied to the setting of principally polarized abelian varieties by substituting the Castelnuovo--Mumford regularity with $\Theta$-regularity of Pareschi--Popa.
\end{abstract}
\section*{Introduction}
In \cite{KP} the authors provide an inequality which ties two important invariants of an ideal sheaf over $\mathbb{P}^n$, namely its Castelnuovo--Mumford regularity and its log--canonical threshold. The main ingredients of their proof are Nadel Vanishing Theorem for multiplier ideals, which works on any smooth projective variety, and Mumford Theorem on Castelnuovo--Mumford regularity (\cite{Laz}*{Theorem 1.8.3}).\par
In a long series of articles (see for example \cites{PP2003,PP2004,PP2008,PP2008b,PP2008c,PP2011,PP3}) Pareschi--Popa developed a regularity theory for abelian varieties which shares many points in common with Castelnuovo--Mumford regularity. In particular, in \cite{PP2003}, they introduced the $\Theta$--regularity index for a sheaf $\mathcal{F}$ on a principally polarized abelian variety $(A,\Theta)$, $\Theta$-reg$(\mathcal{F})$, and they showed some striking similarities of its behavior with that of the Castelnuovo--Mumford regularity index of sheaves on the projective space. For example they prove Theorem \ref{maintheta} below that can be seen as an analogue of the aforementioned Mumford Theorem.\par
In this paper we show that the same inequality proved by K\"uronya and Pintye governs the behavior of the log--canonical threshold and the $\Theta$-regularity index of an ideal sheaf over an principally polarized abelian variety. More precisely, our main result is the following:
\begin{thm}\label{MainResultThetaLCT}
Let (A,$\Theta$) be a principally polarized abelian variety. For any coherent sheaf of ideals $\mathcal{I}\neq \mathcal{O}_A$, the following inequality holds:
$$1 < \text{\emph{lct}}(\I)(\Theta \text{\emph{--reg}}(\I)).$$ 
\end{thm}
This reinforces the idea that there should be a mysterious analogy between the geometry of $\mathbb{P}^n$ and that of principally polarized abelian varieties. This appears already in the work of Debarre \cite{De1995, De1995b} and then was strengthened by Pareschi--Popa (see \cite{PP2008c}*{Section 2} and  \cite{PP2008b}) \par
The paper is organized as follows: in the first section we recall the definition of the multiplier ideal and of the log--canonical threshold of an ideal sheaf on a smooth projective variety. We also give a very brief overview of the theory of $\Theta$-regularity for principally polarized abelian varieties and of continuosly generated sheaves. We then turn to the proof of the main theorem. We first find a lower bound for the $\Theta$-regularity index of nontrivial ideal sheaves on ppav. As a consequence of this we can give a new proof of a celebrated result of Ein--Lazarsfeld on the singularities of pluri-theta-divisors. Then we prove an upper bound for the $\Theta$-regularity index of multiplier ideals, by following a similar argument to the one of \cite{KP}. This together with Nadel Vanishing will allow us to conclude. We end the paper by computing some examples. The firs one, that was suggested to us by Z. Jiang, shows that the inequality provided is sharp.
\subsection*{Notation and conventions}
We work over the field of the complex numbers. Given a morphism $\mu:\widetilde{X}\rightarrow X$ of smooth projective varieties, we denote the relative canonical divisor by 
$$K_{\widetilde{X}/X}:=K_{\widetilde{X}}-\mu^*K_X,$$
where $K_X$ and $K_{\widetilde{X}}$ stand for the canonical classes on $X$ and $\widetilde{X}$ respectively. The round up and round down of rational numbers and $\mathbb{Q}$-divisors are denoted by $\lceil\cdot\rceil$ and $\lfloor\cdot\rfloor$.
\subsection*{Acknowledgements} This project was carried out when the first author was a master student at the University of Bergen under the supervision of the second named author. The first author is grateful to the Department of Mathematics at the UiB for the stimulating learning environment it provided. In addition he would like to thank E. Ferrari, T. Lundemo, G. Scattareggia and M. Vodrup for engaging in many mathematical discussions. Both authors are grateful to M. Gulbrandsen and A. L. Knutsen who were in the final examination committee of the first named author and who provided many comments and suggestions on the thesis on which this article is based. They also want to thanks the unknown referee for spotting a minor problem in the proof of Proposition \ref{ThetaRegBounds}. The second author was partially supported by the grant 261756 of the Research Councils of Norway. She wants also to thank Z. Jiang for suggesting example \ref{ex} (a), L. Lombardi for pointing the paper \cite{LW}, and  Tommaso de Fernex for some very useful correspondence.
\section{Set up and definitions}\label{sec1}

\subsection{Log--canonical thresholds and multiplier ideals} In this section we introduce multiplier ideals and log--canonical thresholds.\par Let $X$ be a smooth projective variety over the complex numbers. Consider $\mathcal{I}$ an ideal sheaf on $X$. A \emph{log--resolution} of $\mathcal{I}$ is a projective birational map $\mu : \widetilde{X}\longrightarrow X$, where $\widetilde{X}$ is a smooth variety, such that the  the inverse image $\mu^{-1}\mathcal{I}\cdot\mathcal{O}_{\widetilde{X}}\simeq\mathcal{O}_{\widetilde{X}}(-E)$ for some effective divisor $E$ and such that the divisor $E + \text{Except}(\mu)$ has simple normal crossing support. Given such a $\mu$ and a nonnegative rational number $c$ we can define the \emph{multiplier ideal sheaf} associated to $\mathcal{I}$ and $c$:
$$\mathcal{J}\left(c\cdot\mathcal{I}\right):=\mu_*\left(\mathcal{O}_{X}(K_{\widetilde{X}/X}-\lfloor cE\rfloor )\right)\subseteq\mathcal{O}_X.$$
As $c$ goes to zero, we see that the multiplier ideal $\mathcal{J}\left(c\cdot\mathcal{I}\right)$ gets bigger and bigger. We define the \emph{log--canonical threshold} of $\mathcal{I}$ to be
$$
\mathrm{lct}(\mathcal{I}):=\inf\{c\in\mathbb{Q}\;|\;\mathcal{J}\left(c\cdot\mathcal{I}\right)\not{\simeq}\mathcal{O}_X\}.
$$

\subsection{Regularity on abelian varieties}In this paragraph we recall the definition of M--regular sheaves on abelian varieties and we present some of their properties.\par Let  $A$ be an abelian variety and denote by $A^\vee$ its dual abelian variety. Given a coherent sheaf $\mathcal{F}$ on $A$ we can consider its \emph{$i$-th cohomological support locus}
$$V^i(A,\mathcal{F}):=\{\alpha\in A^\vee\;|\;H^i(A,\mathcal{F}\otimes\alpha)\neq 0\}.
$$
These are Zariski closed subsets of $A^\vee$. We say that the sheaf $\mathcal{F}$ is \emph{M--regular} if for every positive index $i$ we have that
$$
\mathrm{codim}_{A^\vee}V^i(A,\mathcal{F})>i.
$$
An important and useful property of M--regular sheaves is the following:
\begin{prop}\label{???}
 The Euler characteristic of an M--regular sheaf is always positive.
\end{prop}
\begin{proof}
 Immediate from \cite{Pa2012}*{Lem 1.7 and Lem 1.12 (b)}.
\end{proof}

Turning our attention to a principally polarized abelian variety, $(A,\Theta)$ where $\Theta$ is a symmetric theta-divisor, we say that 
a coherent sheaf $\mathcal{F}$ on $A$  is called \emph{$m$--$\Theta$--regular} if $\mathcal{F}((m-1)\Theta)$ is M--regular. If $m=0$ the sheaf is simply called \emph{$\Theta$--regular}. The following theorem shows that $m$--$\Theta$--regular sheaves behaves under some aspect as Castelnuovo--Mumford regular sheaves on $\mathbb{P}^n$.

\begin{theorem}[{\cite{PP2003} Theorem 6.3}]\label{maintheta}
 Suppose $\mathcal{F}$ is a $\Theta$--regular sheaf on a principally polarized abelian variety $(A,\Theta)$. The following statements hold:
 \begin{enumerate}
\item $\mathcal{F}$ is globally generated;
\item  $\mathcal{F}$ is m--$\Theta$--regular for any $m\geq 1$;
\item the multiplication map $$H^{0}(\mathcal{F}(\Theta))\otimes H^{0}(\mathcal{O}_A(k\Theta))\longrightarrow H^{0}(\mathcal{F}((k+1)\Theta))$$  is surjective whenever $k\geq 2$.
\end{enumerate}
\end{theorem}
Thanks to this result we may introduce the \emph{$\Theta$-regularity index} for a coherent sheaf $\mathcal{F}$ on a principally polarized abelian variety ($A$,$\Theta$): 
$$\Theta\text{--reg}(\mathcal{F}) = \text{inf} \{ m\in \mathbb{Z}\; |\; \mathcal{F} \text{ is } m\text{--}\Theta\text{--regular} \}$$
\subsection{Continuously generated sheaves}
Following Pareschi--Popa \cite{PP2003} we say that a sheaf $\mathcal{F}$ on an irreducible projective variety is \emph{continuously globally generated} (in brief \emph{cgg}), if for every open subsete $U\subseteq\mathrm{Pic}^0(X)$ we have that the following map is surjective
$$
\xymatrix{\bigoplus_{\alpha\in U}\;H^0(X,\mathcal{F}\otimes\alpha)\otimes\alpha^{-1}\ar[rrr]^{\hspace{2cm}\mathrm{ev}_{\mathcal{F}\otimes\alpha}\otimes\mathrm{id}}&&&\mathcal{F}}.
$$
\begin{ex}\label{Mregccg}
 By \cite{PP2003}*{Proposition 2.13} we know that M--regular sheaves are continuosly globally generated.
\end{ex}
In what follows we will need the following characterzations of cgg sheaves, due to Debarre:
\begin{prop}[{\cite{De2006}*{Proposition 3.1}}]\label{Debarre} Let $\mathcal{F}$ be a sheaf on an irriducibile projective varity $X$. Then $\mathcal{F}$ is continuosly globally generated if, and only if, there exist a connected abelian Galois \'etale finite cover $\pi:\tilde X\rightarrow X$ such that $\pi^*(\mathcal{F}\otimes\alpha)$ is globally generated for every $\alpha\in\mathrm{Pic}^0(X)$.
 
\end{prop}

\section{Proof of the main theorem}\label{sec2}
\subsection{Bounds on the $\Theta$-regularity index}
In this paragraph we present some preliminary results which we will need to prove Theorem \ref{MainResultThetaLCT}. In particular we give upper bounds and lower bounds for the $\Theta$--regularity index of an ideal sheaf. We begin with a lower bound:
\begin{prop}\label{geqthree}Let $\mathcal{I}\neq \mathcal{O}_{A}$ be a nonzero ideal sheaf on $A$. Then $\mathcal{I}$ cannot be $m$--$\Theta$--regular for $m<3$. In other words, $\Theta$--reg$(\mathcal{I}) \geq 3$.
\end{prop}
\begin{proof}
By Theorem \ref{maintheta} b), it suffices to prove the statement for the case $m = $ 2. So assume for a contradiction that $\mathcal{I}$ is a 2--$\Theta$--regular ideal sheaf, and consider the associated short exact sequence: $$0\longrightarrow \mathcal{I}\longrightarrow\mathcal{O}_{A}\longrightarrow \mathcal{G}\longrightarrow 0$$ where $\mathcal{G} = \mathcal{O}_A/\mathcal{I}$. Twisting this sequence with $\mathcal{O}_A(\Theta)$ we get
$$0\longrightarrow \mathcal{I} (\Theta)\longrightarrow \mathcal{O}_A(\Theta)\longrightarrow \mathcal{G}(\Theta)\longrightarrow 0.$$ 
Since $\mathcal{I}(\Theta)$ is M-regular by assumption, then we see that the same holds for $\mathcal{G}(\Theta)$. In fact the long exact sequence in cohomology yields that
$$
V^i(A,\mathcal{G}(\Theta))\subseteq V^{i+1}(A,\mathcal{I}(\Theta))\cup V^{i}(A,\mathcal{O}_A(\Theta)),$$
thus we conclude by observing that  the loci $V^{i}(A,\mathcal{O}_A(\Theta))$ are empty for every positive $i$. On the other side, the additivity of the Euler characteristic gives the equality 
\begin{equation}\label{EulerCharAdditive}
1 = \chi (\mathcal{O}_A(\Theta)) = \chi (\mathcal{I}(\Theta)) + \chi (\mathcal{G}(\Theta ))
\end{equation}
We conclude by observing that by Proposition \ref{???} the Euler characteristic of M--regular sheaves is always positive, and so we get an immediate contradiction to \eqref{EulerCharAdditive}. Hence $\mathcal{I}$ cannot be 2--$\Theta$--regular.
\end{proof} 
As a corollary of this lower bound for the $\Theta$-regularity index of a non-trivial sheaf of ideals, we can give a new proof of a celebrated result of Ein-Lazarsfeld:
\begin{theorem}[{Ein--Lazarsfeld \cite{EL19997}}]\label{PluriThetaSing}
Let $(A,\Theta)$ be a principally polarized abelian variety and $m\geq 1$. If we fix any divisor $D\in |m\Theta|$, then $\frac{1}{m}D$ is log--canonical. 
\end{theorem}
\begin{proof}
We want to show that for every rational $\epsilon$, $0<\epsilon<1$ the multiplier ideal associated to the $\mathbb{Q}$ -divisor $\frac{1-\epsilon}{m}D$, $\mathcal{J}\left(\frac{1-\epsilon}{m}D\right)$ (cfr. \cite{Laz}*{Definition 9.2.1}) is $\mathcal{O}_A$. We will do so by proving that its $\Theta$--regularity index is at most 2, contradicting Proposition \ref{geqthree} above.\par
To this aim choose $L = \Theta + F$, where $F$ is a divisor associated to an element $\alpha \in \text{Pic}^{0}(A)$.  We choose $E$ an effective divisor linearly equivalent to $m\Theta$ and $c = \frac{1-\epsilon}{m}$ for a rational number $0 <\epsilon <1$. Then $$L-cE \sim \Theta + F - \frac{1-\epsilon}{m}m\Theta = \epsilon\Theta + F$$  which is an ample $\mathbb{Q}$--divisor and hence nef and big.  For $D\in |m\Theta|$, by Nadel Vanishing for multiplier ideals (\cite{Laz}*{Theorem 9.4.8}) we have that
$$H^{i}\bigg( A,\sO_{A}(\Theta)\otimes \alpha \otimes \mathcal{J}\bigg(\frac{1-\epsilon}{m}D\bigg)\bigg) = 0 \text{ for } i>0.$$ 
In particular $\sO_{A}(\Theta)\otimes \mathcal{J}\big(\frac{1-\epsilon}{m}D\big)$ is an M--regular sheaf (all the higher cohomology support loci are empty) and so we have that  $\mathcal{J}\left(\frac{1-\epsilon}{m}D\right)$ is 2-$\Theta$-regular. Proposition \ref{geqthree} implies that $$\mathcal{J}\left(\frac{1-\epsilon}{m}D\right) \simeq \sO_A,$$ 
and concludes the proof.
\end{proof}
Going back to the $\Theta$--regularity index, we now provide an upper bound for it in the special case of multiplier ideal sheaves:
\begin{prop}\label{ThetaRegBounds}
Let $\mathcal{I}$ be a non--trivial sheaf of ideals with $m = \Theta$--\emph{reg}$(\mathcal{I})$. If $c$ is a rational number $0<c<1$, then we have the following bound
$$\Theta-\mathrm{reg}(\mathcal{J}(c\cdot\mathcal{I})) \leq  \lceil cm\rceil + 1$$
\end{prop}
\begin{proof}
We want to show that $\mathcal{J}(c\cdot\mathcal{I})( \lceil cm\rceil\Theta)$ is M--regular so we will have necessarily that the inequality above holds. To this aim, first observe that our hypotesis ensures that $\mathcal{I}((m-1)\Theta)$ is M--regular and so, by example \ref{Mregccg} continuously globally generated. Then, thanks to Proposition \ref{Debarre}, there exist an abelian varieties $\tilde A$ and an abelian Galois \'etale cover $\pi:\tilde A\rightarrow A$ such that $\pi^*(\mathcal{I}((m-1)\Theta))$ is globally generated.\\
Now we remark that since $\pi$ is \'etale, then the pulback is exact and $\widetilde{\mathcal{I}}:=\pi^*\mathcal{I}$ is an ideal sheaf. Let $D=(\pi^*(m-1)\Theta)$, so we can write $\pi^*(\mathcal{I}((m-1)\Theta))=\widetilde{\mathcal{I}}(D)$ . Choose $\alpha$ in $\mathrm{Pic}^0(A)$ and let $P$ be a divisor representing $\alpha$. Set $L_\alpha=\lceil cm\rceil\pi^*\Theta+\pi^*P$. Observe that $\lceil cm\rceil-c(m-1)\ge c>0,$ so $L_\alpha-cD$ is an ample  divisor. In particular it is nef and big, thus we can apply Nadel Vanishing Theorem for ideal sheaves (\cite{Laz}*{Corollary 9.4.15}) and get that
\begin{equation}\label{IT0}H^{i}\left(\tilde A,\J(c\cdot\widetilde{\I})\otimes \pi^*(\sO_{ A}(\lceil cm\rceil\Theta)\otimes\alpha)\right) =H^{i}\left(\tilde A,\J(c\cdot\widetilde{\I})\otimes \sO_{\tilde A}(L_\alpha)\right)= 0,\end{equation}

\noindent for all $i>0$ and all $\alpha\in\mathrm{Pic}^0(A)$. By pushing forward we get that $\pi_*\J(c\cdot\widetilde{\I})\otimes\sO_{ A}(\lceil cm\rceil\Theta)$ is an M--regular sheaf. Now we can conclude thanks to \cite{Laz}*{Example 9.5.44}, which says that
$$
\mathcal{J}(c\cdot\widetilde{\mathcal{I}})=\mathcal{J}(c\cdot\pi^*\I)=\pi^*\J(A,c\cdot\I)\subseteq\mathcal{O}_{\tilde A}.
$$
Therefore $\J(c\cdot\I)\otimes\sO_{ A}(\lceil cm\rceil\Theta)$ is a summand of $\pi_*\J(c\cdot\widetilde{\I})\otimes\sO_{ A}(\lceil cm\rceil\Theta)$ and hence it is itself M--regular.
\end{proof}
\begin{rmk}
This result could seem in contradiction with Proposition \ref{geqthree} when $c$ is very small: in fact as $c$ goes to 0 we will have $\Theta$-regularity index of $\mathcal{J}\left(c\cdot\mathcal{I}\right)$ is bounded above by 2. On the other side when $c$ is small then $\mathcal{J}\left(c\cdot\mathcal{I}\right)\simeq\mathcal{O}_A$ whose $\Theta$--regularity index is indeed 2.
\end{rmk}

\subsection{The Proof}
 We are now ready to show the proof of Theorem \ref{MainResultThetaLCT}. Set $c = $lct$(\I)$ and note that if $c\geq 1$ then the result follows immediately from Proposition \ref{geqthree}. Otherwise,  by definition of log--canonical threshold we have that $\J(c\cdot\I) \neq \sO_{A}$, and by Propositions \ref{geqthree} and \ref{ThetaRegBounds} we get
$$3\leq \Theta \text{--reg}(\J(c\cdot\I)) \leq \lceil c(\Theta \text{--reg}(\I))\rceil + 1 < c(\Theta \text{--reg}(\I)) + 2$$
which concludes the proof.\qed\par
As a final remark, we point out that the provided inequality is indeed sharp. In fact we will give a sequence of ideals $\mathcal{I}_n$ such that 
$$
\mathrm{lct}(\mathcal{I}_n)\cdot(\Theta-\mathrm{reg}(\mathcal{I}_n))\longrightarrow 1.$$
We compute also other examples.
\begin{ex}\label{ex}(a) (Sharpness)  Let $(A,\Theta)$ a principally polarized abelian variety, and consider $\mathcal{I}_k=\mathcal{O}_A(-k\Theta)$, where $k$ is a positive integer. Then the $\Theta$--regularity index of $\mathcal{I}_k$ is $k+2$, while its log--canonical threshold is $\frac{1}{k}$. So we have that
$$
  \mathrm{lct}(\mathcal{I}_k)\cdot(\Theta-\mathrm{reg}(\mathcal{I}_k))=\frac{k+2}{k}.
  $$
As $k$ goes to infinity we have that $\mathrm{lct}(\mathcal{I}_k)\cdot(\Theta-\mathrm{reg}(\mathcal{I}_k))$ tends to 1 from above.\\
(b) Consider $C\subseteq J(C)$ an Abel--Jacobi embedded curve. Then $C$ is smooth of codimension $g-1$, where $g\geq 2$ denotes the genus of $g$. In \cite{PP2003} Pareschi--Popa computed the $\Theta$--regularity index of $\mathcal{I}_C$ and proved that it is equal to 3. On the other side we can consider the blow-up of $J(C)$ along $C$:
  $$\mu:\widetilde{J}:=\mathrm{Bl}_C(J(C))\rightarrow J(C).$$
  This is a log--resolution of $\mathcal{I}_C$. Denote by $E$ its exceptional divisor. We have that $\mu^{-1}\mathcal{I}_C\cdot\mathcal{O}_{\widetilde{J}}\simeq\mathcal{O}_{\widetilde{J}}(-E)$ while the relative canonical divisor is given by $K_{\widetilde{J}/J(C)}\sim (g-2)E$. Then we have that
  $$\mathrm{lct}(\mathcal{I}_C)=\inf\{c\in\mathbb{Q}\;|\;g-2-\lfloor c\rfloor<0\}=g-1.$$
  We conclude that 
  $$
  \mathrm{lct}(\mathcal{I_C})\cdot(\Theta-\mathrm{reg}(\mathcal{I}_C))=3g-3\geq3
  $$
 (c) Consider $X\subseteq\mathbb{P}^4$ a smooth cubic threefold and let $J(X)$ be its intermediate Jacobian. Then $J(X)$ is a five dimensional abelian variety which admits a principal polarization $\Theta$. The scheme $S$ parametrizing lines in $\mathbb{P}^4$ which are contained in $X$ is a smooth surface (called \emph{Fano surface of lines} of $X$) which admits an embedding $i:S\hookrightarrow J(X)$. In \cite{Ho} H\"oring showed that  $\Theta-\mathrm{reg}(\mathcal{I}_S)=3$. On the other side, a similar argument to that used for Abel--Jacobi embedded curves will show that $\mathrm{lct}(\mathcal{I})=3$. So we get
 $$
  \mathrm{lct}(\mathcal{I_S})\cdot(\Theta-\mathrm{reg}(\mathcal{I}_S))=9
  $$

\end{ex}
We conclude this paper with a couple of open questions left unanswered.

\begin{Question} The $\Theta$--regularity index of Brill--Noether loci $W_d$ inside Jacobians is it known to be 3 (see for example \cite{PP2003}).
 What is the log--canonical threshold of $\mathcal{I}_{W_d}$? More generally in \cite{LW} the authors give a bound for the $\Theta$-regularity of other Brill-Noether loci (for example thos associated to Petri general curves). What is the log--canonical threshold of these loci?
\end{Question}
\begin{Question}The example above shows that when the zero locus of the ideal sheaf $\mathcal{I}$ is reduced, then values much higher of the provided bound are achieved.  Is it possible that 
 $$3 \le \text{\emph{lct}}(\I)(\Theta \text{\emph{--reg}}(\I))$$
 for every ideal sheaf $\mathcal{I}$ on a principally polarized abelian variety whose zero locus is reduced?
\end{Question}


\end{document}